\documentclass[12pt,reqno]{amsart}
\usepackage[all,2cell]{xy}

\usepackage{amssymb}

\newtheorem{theorem}{Theorem}[section]
\newtheorem{proposition}[theorem]{Proposition}
\newtheorem{lemma}[theorem]{Lemma}
\newtheorem{corollary}[theorem]{Corollary}

\theoremstyle{definition}

\newtheorem{remark}[theorem]{Remark}

\numberwithin{equation}{section}

\DeclareMathOperator*{\spec}{Spec}
\DeclareMathOperator*{\Cov}{Cov}

\DeclareMathOperator*{\image}{Im}

\newcommand{\rank}{\operatorname{rank}}
\newcommand{\Hom}{\operatorname{Hom}}
\newcommand{\PP}{\mathbb{P}}
\newcommand{\FF}{\mathbb{F}}

\newcommand{\cO}{\mathcal{O}}
\newcommand{\scrF}{\mathcal{F}}
\newcommand{\scrE}{\mathcal{E}}
\newcommand{\scrV}{\mathcal{V}}

\newcommand{\scrD}{\mathcal{D}}
\newcommand{\scrR}{\mathcal{R}}

\newcommand{\Vect}{\rm Vect}

\begin{document}

\title[Isomorphism of \'etale fundamental groups lifts]{Isomorphism of \'etale fundamental groups lifts to
isomorphism of stratified fundamental group}

\author[I. Biswas]{Indranil Biswas}

\address{Department of Mathematics, Shiv Nadar University, NH91, Tehsil
Dadri, Greater Noida, Uttar Pradesh 201314, India}

\email{indranil.biswas@snu.edu.in, indranil29@gmail.com}

\author[M. Kumar]{Manish Kumar}

\address{Statistics and Mathematics Unit, Indian Statistical Institute,
Bangalore 560059, India}

\email{manish@isibang.ac.in}

\author[A.J. Parameswaran]{A. J. Parameswaran}

\address{Kerala School of Mathematics, Kunnamangalam PO, Kozhikode, Kerala, 673571, India}

\email{param@ksom.res.in}

\subjclass[2010]{14H30, 14G17, 14H60, 14J60}

\keywords{Genuine ramification, fundamental group, stratified bundles}

\begin{abstract}
It is shown that if a finite generically smooth morphism $f\,:\,Y\,\longrightarrow\, X$ of smooth projective
varieties induces an isomorphism of the \'etale fundamental groups, then the induced map of the stratified
fundamental groups $\pi_1^{str}(f)\, :\, \pi_1^{str}(Y,\, y)\,\longrightarrow\, \pi_1^{str}(X,\, f(y))$
is also an isomorphism.
\end{abstract}

\maketitle

\section{Introduction}

Let $X$ be an irreducible smooth projective variety defined over an algebraically closed field $k$ of
characteristic $p\,>\,0$. Gieseker
conjectured that if the \'etale fundamental $\pi_1^{et}(X,\,x)$ is trivial then the stratified 
fundamental group $\pi_1^{str}(X,\,x)$ must be trivial (see \cite{Gi}). This conjecture was proved by Esnault and
Mehta in \cite{EM}. 
Let $f\,:\,Y\,\longrightarrow\, X$ be a morphism of smooth projective varieties. In \cite{ES}, it was shown that if
the induced homomorphism $\pi_1^{et}(f)\, :\, \pi_1^{et}(Y,\,y)\, \longrightarrow\, \pi_1^{et}(X,\,f(y))$
between the \'etale fundamental groups is trivial, then the induced homomorphism $\pi_1^{str}(f)\, :\,
\pi_1^{str}(Y,\,y)\, \longrightarrow\, \pi_1^{str}(X,\,f(y))$ between the stratified fundamental groups
is also trivial. In \cite{Sun} the proof of these two results were simplified.

We prove another relative version of Gieseker conjecture which can be viewed as a generalization of the Gieseker conjecture. The main result of this article is the following:

\begin{theorem} \label{main}
Let $f\,:\,Y\,\longrightarrow\, X$ be a finite generically smooth morphism of smooth projective varieties over
an algebraically closed field $k$ of characteristic $p\,>\,0$ such that the induced homomorphism $
\pi_1^{et}(f)\, :\,
\pi_1^{et}(Y,\, y)\,\longrightarrow\, \pi_1^{et}(X,\, f(y))$ is an isomorphism.
Then the induced homomorphism $\pi_1^{str}(f)\, :\,\pi_1^{str}(Y)\,\longrightarrow\, \pi_1^{str}(X)$ of
stratified fundamental groups is also an isomorphism.
\end{theorem}

Note that Theorem \ref{main} implies the conjecture of Gieseker. Indeed, if $\pi_1^{et}(X)$ is trivial,
then consider a finite generically smooth morphism $f\,:\, X\,\longrightarrow\, \PP^d$,
where $d\,=\, \dim X$, which can be constructed using Noether normalization;
since $\pi_1^{et}(\PP^d)$ and $\pi_1^{str}(\PP^d)$ are trivial, Theorem \ref{main} implies
that $\pi_1^{str}(X)$ is trivial.

Note that in the earlier results (\cite{EM}, \cite{ES}, \cite{Sun}), one only needs to show that certain
stratified bundles are trivial bundles. In our work, we show that the pullback functor between the category
of stratified bundles on $X$ to that of $Y$ is an equivalence of categories.

The fact that the functor is fully faithful is proved using Proposition \ref{fullyfaithful} (a generalization of 
\cite[Lemma 4.3]{BP}). Note that this only requires the map $f$ to be genuinely ramified. Hence we also get 
$\pi_1^{et}(f)$ is surjective implies $\pi_1^{str}(f)$ is also surjective (see Remark \ref{surjective}).

Other key ingredients in the proof of Theorem \ref{main} are Theorem \ref{etale-main} (which may be of independent 
interest) Lemma \ref{Hrus} (a result of Hrushovski) and the theory of representation spaces developed by Simpson 
(\cite{Simpson}) and its generalization to positive characteristics (\cite{Sun}).

\section{Stratified vector bundles}

Let $k$ be an algebraically closed field of characteristic $p$, with $p\,>\, 0$. Let $X$ be an irreducible
smooth projective variety over $k$. Denote by ${\mathcal D}_X$ the sheaf of differential operators,
in the sense of Grothendieck, on $X$ \cite{Gr}, \cite{BO}. A stratified vector bundle on $X$ is an
${\mathcal O}_X$--coherent ${\mathcal D}_X$--module \cite{Gi}.

Let $F_X\, :\, X\, \longrightarrow\, X$ be the absolute Frobenius morphism for $X$. So for any vector bundle $V$ on
$X$, the pullback $F^*_X V\, \longrightarrow\, X$ is the subbundle of $V^{\otimes p}$ defined by
$\{v^{\otimes p}\, \in\, V^{\otimes p}\,\,\big\vert\,\, v\, \in\, V\}$. A $F$--divisible vector bundle
on $X$ is a sequence of vector bundles $\{E_i\}_{i\geq 0}$ on $X$ indexed by the nonnegative integers together with an
isomorphism $E_i\, \longrightarrow\, F^*_X E_{i+1}$ for every $i\, \geq\, 0$ \cite{Gi}, \cite{Sa}.
There is a natural equivalence of categories between the stratified vector bundles and the $F$--divisible
vector bundles. The underlying vector bundle for the stratified bundle corresponding to a $F$--divisible
vector bundle $\{E_i\}_{i\geq 0}$ is $E_0$ \cite{Gi}, \cite{Sa}. Similarly, the rank of $\{E_i\}_{i\geq 0}$ is
the common rank of $E_i$.

Let $\Vect^{str}(X)$ be the category of $F$--divisible vector bundles on $X$, which, as mentioned above, is 
equivalent to the category of stratified vector bundles on $X$. Henceforth, by a stratified bundle we will mean a 
$F$--divisible vector bundle. The category of finite dimensional $k$--vector spaces will be denoted by $\Vect(k)$. 
Fix a closed point $x$ of $X$. We have the fiber functor
\begin{equation}\label{e1}
\omega_x\,:\, {\rm Vect}^{str}(X)\, \longrightarrow\, \Vect(k),\ \ \{E_i\}_{i\geq 0}\, \longmapsto\,
(E_0)_x.
\end{equation}
Then the pair $(\Vect^{str}(X),\, \omega_x)$ forms a neutral Tannakian category.
Its Tannaka dual is the stratified fundamental group $\pi_1^{str}(X,\, x)$ \cite{Sa} \cite{Gi}
(see \cite{SR}, \cite{DM}, \cite{No1}, \cite{No2} for Tannaka dual).

Let $\Vect^{et}(X)$ denote the category of \'etale trivializable vector bundles on $X$. Note that any
\'etale trivial vector bundle on $X$ gives rise to a stratified bundle and this induces a fully faithful
functor from $\Vect^{et}(X)$ to $\Vect^{str}(X)$. This functor Induces an epimorphism $\pi_1^{str}(X,\,x)
\,\longrightarrow\, \pi_1^{et}(X,\,x)$.

Let $f\,:\,Y\,\longrightarrow\, X$ be a generically smooth morphism of irreducible smooth projective
varieties over $k$. Fix a closed point $y$ of $Y$, and set $x\,=\,f(y)$. Then the pullback functor 
$$f^*\,:\,({\Vect}^{str}(X),\,\omega_x) \,\longrightarrow\, ({\Vect}^{str}(Y),\,\omega_y)$$
is a functor of Tannakian categories, and it induces a homomorphism of group schemes
$$\pi_1^{et}(f)\, :\, \pi_1^{str}(Y,\, y)\,\longrightarrow\, \pi_1^{str}(X,\, x)$$
between the corresponding Tannaka duals.

Let $\cO_{X}(1)$ be a fixed very ample line bundle $X$. For a torsionfree sheaf $V$ of $X$, $$P(V,m)\ :=\ \chi(V(m))$$ 
is a polynomial in $m$ which is called the Hilbert polynomial of $V$. We say that $V$ is semistable if for
any nonzero subsheaf $W\,\subset\, V$, the inequality
$$\frac{P(W,m)}{\rank(W)}\ \le\ \frac{P(V,m)}{\rank(V)} 
$$
holds for all $m$ sufficiently large.

Let $\scrE\,=\,(E_n)_{n\ge 0}$ be a stratified vector bundle on $X$ of rank $r$. For all $n\, \geq\, 0$,
the Hilbert polynomial of $E_n$ is same as that of the trivial vector bundle $\cO_X^{\oplus r}$, and
there exists an integer $n_0\,\ge\, 1$ (which depends on
$\scrE$) such that for all $j\,\ge\, n_0$, the vector bundle $E_j$ is 
semistable. This is because if $W$ is a subsheaf of $E_n$, then $(F_X^n)^* W$ is subsheaf of
$(F_X^n)^* E_n\,=\, E_0$; on the other hand, we have $c_i((F_X^n)^* W)\,=\, p^{ni} c_i(W)$. Now from the
boundedness of the destabilizing subsheaves of $E_0$ it follows that $E_n$ is semistable for sufficiently
large $n$. This also shows that $c_i(E_j)\,=\, 0$ for all $i\, \geq\, 1$ and all $j\, \geq\, 0$.

Let $X$ be a smooth irreducible projective variety over $k$. Fix a closed
point $$\xi\,:\,\spec(k)\,\longrightarrow\, X$$ of $X$. Recall from \cite{Sun} that a representation space
$\scrR(X,\,\xi,P)$ parametrizes all pairs $(V,\, \beta)$ 
where $V$ is a semistable vector bundle with Hilbert polynomial $P$ and $\beta\,:\,\xi^*V\,\longrightarrow\, 
\cO_{\spec(k)}^{\oplus r}$ is an isomorphism. This was constructed by Simpson in characteristic zero and it was 
extended to positive characteristics by Sun. In particular, in \cite[Theorem 2.3]{Sun} it was shown that 
$\scrR(X,\xi,P)$ is in fact a fine moduli space.

\begin{proposition}\label{pullback-morphism}
Let $f\,:\,Y\,\longrightarrow\, X$ be a finite generically smooth morphism. Let $\zeta$ be a closed point
in $Y$, and $\xi \,=\, f(\zeta)$. Then $f$ induces a morphism $\Phi\,:\,\scrR(X,\xi,P_X)\,\longrightarrow\,
\scrR(Y,\zeta,P_Y)$ defined by $(V,\,\beta)\,\longmapsto\, (f^*V,\, f^*\beta\big\vert_{\zeta})$,
where $P_X$ and $P_Y$ are the Hilbert polynomials of $\cO_X^{\oplus r}$ and $\cO_Y^{\oplus r}$ respectively.
\end{proposition}

\begin{proof}
Since $f$ is finite generically smooth morphism, for a semistable vector bundle $V$ on $X$ with Hilbert polynomial 
$P_X$, the pullback $f^*V$ is a semistable vector bundle of Hilbert polynomial $P_Y$ on $Y$. Denote 
$\scrR\,=\,\scrR(X,\xi,P)$, and let $f_{\scrR}$ be the base change of $f$ to $\spec(\scrR)$, so
$$
f_{\scrR}\,=\, f\times {\rm Id}_{\spec(\scrR)}\, :\, Y\times \spec(\scrR)\, \longrightarrow\,
X\times \spec(\scrR).
$$
Let $(\scrV,\,\beta_{\scrR})\, \longrightarrow\, X\times \scrR$ be the universal vector bundle.
Then $f_{\scrR}^*\scrV$ is a semistable vector bundle on $Y_{\scrR}\times \spec(\scrR)$. The
pair $(f_{\scrR}^*\scrV,\, f_{\scrR}^*\beta_{\scrR}\big\vert_{\zeta\times\scrR})$ is
a family of semistable vector bundles on $Y$ together with 
an isomorphism of the fiber over $\zeta$ with $k^{\oplus r}$. Hence by universal property of moduli spaces
there is a morphism $\Phi\,:\,\scrR\,\longrightarrow\, \scrR(Y,\zeta, P_Y)$ as in the statement of
the proposition.
\end{proof}

The following is a consequence of \cite[Corollary 1.2]{Hru} (also see \cite[Corollary 0.4]{Var}).

\begin{lemma} \label{Hrus}
Take an irreducible variety defined $Z$ over $\overline \FF_p$, and let $\Psi\,:\,Z\,\longrightarrow\, Z$ be
a rational dominant map. Then the subset $$S\, :=\, \{z\,\in\, Z\,\,
\big\vert\,\, \Psi^n(z)\,=\,z\,\, \text{ for some }\,\, n \}\, \subset\, Z$$ is dense in $Z$.
\end{lemma}

A finite generically smooth map $f\,:\,Y\,\longrightarrow\, X$ of two irreducible
projective varieties of the same dimension
is called \textit{genuinely ramified} if the induced homomorphism of \'etale fundamental groups
$$\pi^{et}(f)\,:\,\pi_1^{et}(Y,\, y)\,\longrightarrow\, \pi_1^{et}(X, \, f(y))$$
is surjective.

The following proposition is proved in Section \ref{ffp}.

\begin{proposition}\label{fullyfaithful}
Let $f\,:\,Y\,\longrightarrow\, X$ be a genuinely ramified map between irreducible smooth projective
varieties. Let $V$ and $W$ be semistable vector bundles on $X$ of same slope. Then the natural
map $$\Hom_X(V,\,W)\,\longrightarrow\, \Hom_Y(f^*V,\,f^*W)$$
is an isomorphism. 
\end{proposition}

We will assume Proposition \ref{fullyfaithful} and defer its proof to Section \ref{ffp}.

\section{The case of $k=\overline{\FF}_p$}
In this section we prove the main theorem when $k\,=\,\overline{\FF}_p$.
\begin{theorem}\label{main-basecase}
Let $f\,:\,Y\,\longrightarrow\, X$ be a genuinely ramified map of smooth projective varieties
defined over $\overline{\FF}_p$. If the induced homomorphism $\pi^{et}(f)\,:\,\Vect^{et}(X)\,\longrightarrow\,
\Vect^{et}(Y)$ is essentially surjective, then so is the induced homomorphism $\pi_1^{str}(f)\, :\,\Vect^{str}(X)
\,\longrightarrow\, \Vect^{str}(Y)$.
\end{theorem}

\begin{proof}
Let $\scrE\,=\,(E_n)_{n\ge 0}$ be a stratified vector bundle on $Y$ of rank $r$. As noted before,
there exists an integer $m$ such that for all
$n\,\ge\, m$, the vector bundle $E_n$ is semistable of rank $r$, and also $c_i(E_j)\,=\, 0$
for all $i\, \geq\, 1$ and $j\, \geq\, 0$.
Let $\eta\,\in\,Y$ be a closed point. For every $n \,>\, m$,
by \cite[Lemma 3.3]{Sun} there is an isomorphism $$\beta_n\,:\,\eta^*E_n\,\longrightarrow\, \cO_{\spec(k)}^{\oplus r}$$
such that $F^*_Y\beta_{n+1}\,=\,\beta_n$.
Let $N$ be the closure of
\begin{equation}\label{k1}
T\,:=\, \{(E_n,\,\beta_n)\,\,\big\vert\,\, n \,>\, m \}
\end{equation}
in the representation space $\scrR(Y,\eta, P_Y)$ parametrizing the isomorphism classes of
pairs $(V,\,\beta)$, where $V$ is a vector bundle on $Y$ with Hilbert polynomial $P_Y$ while
$\beta\,:\,\eta^*V\,\longrightarrow\, \cO_{\spec(k)}^{\oplus r}$ is an isomorphism. From
\cite[Theorem 2.3 (1)]{Sun} it follows that $(E_n,\,\beta_n)$ is a point of $\scrR(Y,\eta, P_Y)$ for
every $n \,>\, m$.
 
In \cite[Proposition 2.5]{Sun} it is shown that the analog of the rational map Verschiebung exists on
$\scrR(Y,\eta, P_Y)$. Let
\begin{equation}\label{k2}
{\mathcal V}\,:\,N\,\longrightarrow\, N
\end{equation}
be the restriction of this Verschiebung. Note that the image of
$\mathcal V$ contains $T$ (see \eqref{k1}),
and hence ${\mathcal V}\,:\,N\,\longrightarrow\, N$ is a dominant rational map.

Let $N_1,\,\cdots,\, N_\ell$ be the positive dimensional irreducible components of $N$. The generic point
of $N_j$, $1\,\leq\, j\, \leq\, \ell$, will be denoted by $\eta_j$. Since the map ${\mathcal V}$ in
\eqref{k2} is dominant, it permutes $\eta_1,\,\cdots,\,\eta_\ell$. Hence there exists $a\,\ge \,1$ such that
${\mathcal V}^a$ fixes $\eta_i$ for all $1\,\le\, i \,\le\, \ell$.

Recall that each $N_i$ is irreducible. Hence by Lemma \ref{Hrus}, the set
$$S_i\,=\,\{(E,\,\beta)\,\in\, N_i\,\, \big\vert\,\, {\mathcal V}^b((E,\,\beta))\,=\,(E,\,\beta)\, \text{ for some }\,
b \,\ge\, 1\}$$ is dense in $N_i$. Consequently,
$$S\,=\,\{(E,\,\beta)\,\in\, N\,\,\big\vert\,\, {\mathcal V}^b((E,\,\beta))\,=\,(E,\,\beta)\, \text{ for some }
\, b \,\ge\, 1\}$$
is dense in $N$.
 
For any $(E,\,\beta)\,\in\, S$, we have $E$ to be a semistable vector bundle of rank $r$ on $Y$ such
that $(F^j_Y)^*E\,=\, E$ for some $j\, \geq\, 1$. Hence a theorem of
Lange and Stulher says that $E$ is an \'etale trivial vector bundle on $Y$ \cite{LS}. Since
$f^*\,:\,\Vect^{et}(X)\,\longrightarrow\, \Vect^{et}(Y)$ is essentially surjective,
$$E\ =\ f^*E'$$ for some \'etale trivial vector
bundle $E'$ on $X$. Moreover $E'$ is a semistable vector bundle of rank $r$
whose Hilbert polynomial is the Hilbert polynomial of $\cO_{X}^{\oplus r}$.

Consider the morphism
$$\Phi\,:\,\scrR(X,f(\eta),P_X)\,\longrightarrow\, \scrR(Y,\eta,P_Y)$$
defined by the pullback using $f\,:\,Y\,\longrightarrow\, X$ as in Proposition \ref{pullback-morphism}.
Let $N'\,=\,\Phi^{-1}(N)$. Note that the restriction $\Phi\big\vert_{N'}\,:\,N'\,\longrightarrow\, N$ is
dominant, because $S$ is in $\image(\Phi)$ and $S$ is dense in $N$.
Consequently, there exists an open dense subset $U$ of $N$ contained in $\image(\Phi)$.
As $T$ is dense in $N$, there exists a subsequence $\{n_j\,\, \big\vert\,\, n_j \,>\, m\}$ such that
$\{(E_{n_j},\,\alpha_j)\,\, \big\vert\,\, j\,\ge\, 1\}\,\subset\, U$.
Consequently, there are semistable vector bundles $E'_{n_j}$ on $X$ such that $f^*E'_{n_j}\,=\,E_{n_j}$.
We have $f\circ F_Y\,=\, F_X\circ f$, and therefore, $$E_{n_j-1}\,=\,F^*_YE_{n_j}\,=\,F^*_Yf^*E'_{n_j}
\,=\,f^*F^*_XE'_{n_j}.$$
Set $n_0\,=\,0$, and define $E'_{i}\,:=\,(F^{n_j-i}_X)^*E'_{n_j}$, where
$j\,\ge \,0$ and $n_{j-1}\,<\, i \,<\, n_j$.
Then for all $n$, we have $E_n\,=\,f^*E'_n$ for some vector bundles $E'_n$ on $X$. 
Also for $j\,\ge\, 0$, $$f^*F^*_XE'_{n_j+1} \,\cong\, f^*E'_{n_j}.$$ But $f$ is genuinely ramified and for
$j\,\ge\, 1$, the vector bundle $E'_{n_j}$ is semistable. Consequently, we have $F^*_X E'_{n_j+1}\,\cong\,
E'_{n_j}$ by Proposition \ref{fullyfaithful}. Therefore, $\scrE'\,=\,(E'_{n_j})_{j\ge 0}$ is a stratified bundle
on $X$ and it satisfies the condition $f^*\scrE'\,=\,\scrE$.
\end{proof}

\begin{corollary}\label{basecase}
Let $f\,:\,Y\,\longrightarrow\, X$ be a finite generically smooth morphism of smooth projective varieties
over $\overline{\FF}_p$. Let $y\,\in\, Y$ be a closed point, and $x\,=\,f(y)$. It is given that the induced homomorphism
$\pi^{et}(f)\, :\, \pi_1^{et}(Y,\,y)\,\longrightarrow\, \pi_1^{et}(X,\,x)$ is an isomorphism. Then the induced
homomorphism $\pi^{str}(f)\, :\, \pi_1^{str}(Y,\,y)\,\longrightarrow\, \pi_1^{str}(X,\,x)$ is also an isomorphism.
\end{corollary}

\begin{proof}
The hypothesis implies that $f$ is genuinely ramified, and the pullback functor 
$f^*\,:\,\Vect^{et}(X)\,\longrightarrow\, \Vect^{et}(Y)$ is an equivalence of categories. By Theorem 
\ref{main-basecase}, the pullback functor on stratified bundles $f^*\,:\, \Vect^{str}(X)\,\longrightarrow\, 
\Vect^{str}(Y)$ is essentially surjective.

Let $g\,:\,(V_n)_{n\ge 0}\,\longrightarrow\, (W_n)_{n\ge 0}$ be a morphism in the category $\Vect^{str}(X)$. Then by 
definition, $g$ consist of morphisms of vector bundles $$g_n\,:\,V_n\,\longrightarrow\, W_n$$ such that 
$F^*_X(g_{n+1})\,=\,g_n$ for all $n\,\ge\, 0$. Therefore, $g$ is
uniquely determined by $\{g_n\,\,\big\vert\, n\, \text{ sufficiently large}\}$.
Note that there exists an integer
$m$ such that for $n\,\ge\, m$, the vector bundles $V_n$ and $W_n$ are semistable. Since $f$ is genuinely
ramified, by Proposition \ref{fullyfaithful}, $$\Hom_X(V_n,\,W_n)\,\cong\, \Hom_Y(f^*V_n,\,f^*W_n)$$ for all $n\,\ge\, m$.
Hence $f^*\,:\, \Vect^{str}(X)\,\longrightarrow\, \Vect^{str}(Y)$ is also fully faithful. Since $f^*$ is a tensor
functor which is 
an equivalence of Tannakian categories, the induced homomorphism between the Tannakian duals is an isomorphism.
\end{proof}

\section{The general case}

In this section, we are given a family of morphisms $f_T\,:\, Y_T\,\longrightarrow\, X_T$ of irreducible smooth
projective varieties parametrized by an integral $k$--scheme $T$ such that the restriction of $f_T$ over the 
generic geometric point of $T$ is genuinely ramified (respectively, the restriction of $f_T$ induces an
isomorphism of \'etale fundamental groups). We show that the fiber $f_t$ for any $t$ in
an open dense subset of $T$ is genuinely ramified (respectively, $f_t$ induces an
isomorphism of \'etale fundamental groups). See Theorem \ref{etale-main} for the precise statement.

For a connected scheme $X$, let
\begin{equation}\label{y1}
\Cov(X)
\end{equation}
denote the category of finite \'etale covers of $X$, so the objects of $\Cov(X)$
are finite \'etale morphisms $Z\,\longrightarrow\, X$, and the morphisms from an object $Z\,\longrightarrow\, X$ to 
another object $Z'\,\longrightarrow\, X$ are the $X$-morphisms $Z\,\longrightarrow\, Z'$. Note that a morphism
of varieties $f\,:\, 
Y\,\longrightarrow\, X$ induces a pullback functor $\Cov(X)\,\longrightarrow\, \Cov(Y)$ given by the base change from 
$X$ to $Y$.

\begin{proposition} \label{etale-local}
Let $R$ be a complete discrete valuation ring with algebraically closed residue field $k$, and denote by $K^s$
the separable closure of the fraction field of $R$. Set $T\,=\,\spec(R)$, and let $\eta\,:\,\spec(K^s)\,
\longrightarrow\, T$ be the generic geometric point of $T$. Let $$f_T\ :\ Y_T\ \longrightarrow\ X_T$$ be a finite
flat generically smooth morphism of smooth proper 
integral $T$--schemes $X_T$ and $Y_T$ such that the closed fiber $$f_0\ :\ Y_0\ \longrightarrow\ X_0$$ is also a
finite flat generically smooth morphism of smooth connected proper schemes over $k$. Let
$$f_{\eta}\ :\ Y_{\eta}\ \longrightarrow\ X_{\eta}$$ be the fiber of $f_T$ over the geometric generic point. If 
$f_{\eta}$ is genuinely ramified, then so is $f_0$.

Let $y_0$ (respectively, $y$) be a geometric point of $Y_0$ (respectively, $Y_{\eta})$, with $x_0\,:=\,f_0(y_0)$ and 
$x\,:=\,f_\eta(y)$. Also, assume that $X_0$ is not contained in the branch locus of $f_T$. If the homomorphism 
$$\pi_1^{et}(f_\eta)\ :\ \pi_1^{et}(Y_\eta,\,y)\ \longrightarrow\ \pi_1^{et}(X_\eta,\,x)$$
induced by $f_{\eta}$ is an isomorphism, then the induced homomorphism
$$\pi_1^{et}(f_0)\ :\ \pi_1^{et}(Y_0,\,y_0)\ \longrightarrow\ \pi_1^{et}(X_0,\,x_0)$$ is an isomorphism.
\end{proposition}

\begin{proof}
We first note that the specialization map $\pi_1^{et}(X_{\eta},\,x)\,\longrightarrow\, \pi_1^{et}(X_0,\,x_0)$ is
surjective (\cite{SGA}, \cite[9.2]{Mur}). Also, by the functoriality of $\pi_1^{et}$, the following diagram 
is commutative:
\[\xymatrix{
\pi_1^{et}(Y_{\eta},\,y)\ar[r]\ar[d] & \pi_1^{et}(X_{\eta},\,x)\ar[d]\\
\pi_1^{et}(Y_0,\,y_0)\ar[r] & \pi_1^{et}(X_0,\,x_0).
 }
 \]
Therefore, if $f_{\eta}$ is genuinely ramified, then the homomorphism
$$\pi_1^{et}(Y_{\eta},\,y)\,\longrightarrow\, \pi_1^{et}(X_0,\,x_0)$$
induced by $f_{\eta}$ is also surjective. Consequently,
the map $\pi_1^{et}(f_0)\,:\,\pi_1(Y_0)\,\longrightarrow\, \pi_1(X_0)$ is also surjective, i.e., the map
$f_0$ is genuinely ramified. This is also equivalent to the statement that the pullback functor
$\Cov(X_0)\,\longrightarrow\, \Cov(Y_0)$ (see \eqref{y1}) is fully faithful (\cite{SGA}, \cite[Lemma 58.4.1]{SP}).
 
Now to prove that $\pi_1^{et}(f_0)\,:\,\pi_1^{et}(Y_0)\,\longrightarrow\, \pi_1^{et}(X_0)$ is an isomorphism under the 
additional hypothesis that $\pi_1^{et}(f_\eta)$ is an isomorphism, it is enough to show that the pullback
functor $\Cov(X_0)\,\longrightarrow\, \Cov(Y_0)$ is essentially surjective.
 
There is an equivalence of categories $\Cov(Y_T)\,\stackrel{\sim}{\longrightarrow}\, \Cov(Y_0)$ (see \cite[Lemma 58.9.1]{SP} or 
\cite{SGA}). Let $Z_0\,\longrightarrow\, Y_0$ be a finite \'etale connected cover. This induces a finite \'etale
connected covering
$Z_T\,\longrightarrow\, Y_T$. Since $\pi_1(Y_{\eta})\,\longrightarrow\, \pi_1(X_{\eta})$ is an isomorphism, the 
functor $\Cov(X_{\eta})\,\longrightarrow\, \Cov(Y_{\eta})$ is an equivalence of categories. So we conclude that 
$Z_{\eta}\,\longrightarrow\, Y_{\eta}$ is the pull-back of an \'etale connected covering
$W_{\eta}\,\longrightarrow\, X_{\eta}$.
 
There exists a finite separable extension $K/k(T)$ such that $$Z_T\times_T \spec(K)\ \longrightarrow\
Y_T\times_T \spec(K)$$ is the pullback of an \'etale connected cover $W\,\longrightarrow\, X_T\times_T \spec(K)$
with $$W\times_{\spec(K)}\spec(K^s)\,=\,W_{\eta}.$$ Since $X_T$ is integral, it
follows that $W$ is also integral. The normalization of $T$ in $K$ will be denoted by
$\widehat T$. Denote by
$X_{\widehat T}$, $Y_{\widehat T}$ and $Z_{\widehat T}$ the base change --- to $\widehat T$ --- of $X_T$,
$Y_T$ and $Z_T$ respectively. Let $W_{\widehat T}\,\longrightarrow\, X_{\widehat T}$ be the normalization of
$X_{\widehat T}$ in $k(W)$. So we get the following commutative diagram:
\[\xymatrix{
Z_{\widehat T}\ar[r]\ar[d] & W_{\widehat T}\ar[d]\\
Y_{\widehat T}\ar[r] & X_{\widehat T}.
}
\]
Note that $Z_T\,\longrightarrow\, Y_T$ is \'etale, and the branch locus of $Y_T\,\longrightarrow\, X_T$ does not 
contain $X_0$. Consequently, the 
branch locus of $Z_T\,\longrightarrow\, X_T$ does not contain $X_0$. Therefore, the branch 
locus of $W_{\widehat T}\,\longrightarrow\, X_{\widehat T}$ does not contain $X_0$; moreover,
this map is \'etale over 
the generic point of $\widehat T$. Thus by the purity of branch locus (\cite[Exp. X, Thm. 3.1]{SGA}, \cite[Lemma 
53.20.4]{SP}), the map $W_{\widehat T}\,\longrightarrow\, X_{\widehat T}$ is \'etale, and
therefore this map restricted to the 
closed fiber $W_0\,\longrightarrow\, X_0$ is also \'etale.
Note that the pullback of $W_{\widehat T}\,\longrightarrow\, 
X_{\widehat T}$ to $Y_{\widehat T}$ is $Z_{\widehat T}\,\longrightarrow\, Y_{\widehat T}$. Thus the pullback of 
$W_0\,\longrightarrow\, X_0$ to $Y_0$ is $Z_0\,\longrightarrow\, Y_0$. This completes the proof of
the statement that the pullback 
functor $\Cov(X_0)\,\longrightarrow\, \Cov(Y_0)$ is essentially surjective.
As mentioned before, this implies that
$\pi_1^{et}(f_0)\,:\,\pi_1^{et}(Y_0)\,\longrightarrow\, \pi_1^{et}(X_0)$ is an isomorphism. This
completes the proof of the proposition.
\end{proof}

\begin{theorem}\label{etale-main}
Let $k$ be an algebraically closed field, and let $T$ be a finite type connected integral scheme defined over $k$.
Denote by $K^s$ 
the separable closure of the function field of $T$, and let $\eta\,:\,\spec(K^s)\,\longrightarrow\, T$ be the generic 
geometric point of $T$. Let $f_T\,:\,Y_T\,\longrightarrow\, X_T$ be a finite generically smooth morphism of proper 
integral smooth schemes over $T$ such that the morphism $f_{\eta}\,:\,Y_{\eta}\,\longrightarrow\, 
X_{\eta}$ is genuinely ramified. Then there is an open dense subset $U$ of $T$ such that for all closed points $t\,
\in\,U$, the morphism of the fibers over $t$, $$f_t\,:\,Y_t\,\longrightarrow\, X_t,$$ is a
genuinely ramified map of smooth proper varieties.
 
Moreover, if $f_{\eta}$ induces an isomorphism of the \'etale fundamental groups, then there is open dense subset 
$U'$ of $T$ such that for every closed points $t\,\in\, U'$, the morphism of the fibers over $t$,
$$f_t\,:\,Y_t\,\longrightarrow\, X_t,$$ induces an isomorphism
$\pi_1^{et}(Y_t) \,\stackrel{\sim}{\longrightarrow}\, \pi_1^{et}(X_t)$
of the \'etale fundamental groups.
\end{theorem}

\begin{proof}
By replacing $T$ by an open dense subscheme, we may assume that $f_T$ is flat, and for all closed points $t\,\in\, T$, 
$$f_t\,:\,Y_t\,\longrightarrow\, X_t$$ is a finite flat generically smooth morphism of smooth proper varieties. Also,
since $f_T$ is generically \'etale, the branch locus of $f_T$ is a proper closed subscheme of $X_T$. Consequently, by 
shrinking $T$ further, if necessary, we may assume that for all $t\,\in\, T$, the fiber $X_t$ is not contained in
the branch locus of $f_T$.
 
Fix a closed point $t\,\in\, T$. Let $R$ be the completion of a discrete valuation
ring with fraction field $k(T)$ dominating 
$\cO_{T,t}$. Denote $\widehat T\,=\,\spec(R)$, and let $f_{\widehat T}\,:\,Y_{\widehat T}\,\longrightarrow\, 
X_{\widehat T}$ be the pullback of $f_T$ along the natural morphism
$\widehat T\,\longrightarrow\, T$. Denote by $\widehat K$ the 
separable closure of the fraction field of $R$. Then $K^s$ is a subfield of $\widehat K$. Note that we have
$$\pi_1^{et}(X_{\eta})\,\,\cong\,\, \pi_1^{et}(X_{\eta}\otimes_{K^s}\widehat K)$$ (\cite{SGA},
\cite[Proposition 7.3.2]{Mur}). Now the given condition that $Y_{\eta}\,\longrightarrow\, X_{\eta}$ is genuinely 
ramified implies that $Y_{\widehat K}\,\longrightarrow\, X_{\widehat K}$ is also genuinely ramified. Hence by 
Proposition \ref{etale-local}, $Y_t\,\longrightarrow\, X_t$ is genuinely ramified.

Finally, if $Y\,\longrightarrow\, X$ induces an isomorphism of \'etale fundamental groups, then $$Y\otimes 
\widehat{K}\ \longrightarrow\ X\otimes \widehat K$$ also induces an isomorphism $\pi_1^{et}(Y_{\widehat 
K})\,\longrightarrow\, \pi_1^{et}(X_{\widehat K})$. Now apply Proposition \ref{etale-local} to conclude that 
$\pi_1^{et}(Y_t)\,\longrightarrow\, \pi_1^{et}(X_t)$ is an isomorphism.
\end{proof}

\section{Proof of Theorem \ref{main}}

Let $f\,:\,Y\,\longrightarrow\, X$ be a finite separable morphism of smooth projective varieties defined
over an algebraically 
closed field $k$ of characteristic $p\,>\,0$. Let $\xi$ be a closed point of $Y$. Let $\scrR(X,f(\xi),P_X)$ and 
$\scrR(Y,\xi, P_Y)$ be the representation spaces over $X$ and $Y$ respectively, as in the proof of Theorem 
\ref{main-basecase}. Let $\scrE\,=\,\{E_n\}_{n\ge 0}$ be a stratified vector bundle on $Y$.

There exists a finite type smooth integral $\overline \FF_p$--scheme $T$ with function field $\FF_p(T)\,\subset\, K$
satisfying the following three statements:
\begin{enumerate}
 \item There is a finite generically smooth $T$--morphism $f_T\,:\,Y_T\,\longrightarrow\, X_T$ such that the
base change of $f_T$ to $k$ is $f$.

\item The representation spaces $\scrR(X,f(\xi),P_X)$ and $\scrR(Y,\xi, P_Y)$ are both defined over $k(T)$.

\item There exists a stratified vector bundle $\scrE_T\,=\,\{E_{T,n}\}_{n\ge 1}$ on $Y_T$ such that the
pullback of $\scrE_T$ along the morphism $Y\,\longrightarrow\, Y_T$ is $\scrE$.
\end{enumerate}
This is because of the following:
Since $X$, $Y$, $f$, $\scrE$ and the representation spaces are all of finite type, there exist finitely many 
elements $z_1,\,\cdots,\, z_m \,\in\, k$ such that $Y$, $X$, $f$, $\scrR(X,f(\xi),P_X)$, $\scrR(Y,\xi, P_Y)$ and $\scrE$ 
are all defined over $\spec(\overline\FF_p[z_1,\cdots,z_m])$. Take $T$ to be a smooth affine open subset of 
$\spec(\overline\FF_p[z_1,\cdots,z_m])$.

Here we view $\scrE$ as a locally free $\cO_Y$--coherent ${\mathcal D}_Y$--module. For any integer $n\, \geq\, 1$, let $\scrD_Y ^{< n}$ denote the
subsheaf of $\scrD_Y$ of differential operators of order strictly less than $n$. Note that for any open $U\,\subset\, Y$ and $n\,\ge\, 0$,
\[
E_n(U)\,=\, \{s\,\in\, \scrE(U)\,\,\big\vert\,\, D(s)\,=\,0\ \ \forall\ \, D\,\in\, \scrD^{<p^n}_Y\, \text{ with }\, D(1)\,=\,0\}
\]
\cite[Theorem 1.3]{Gi}. We obtain a sequence of subsheaves
$$\scrE\,=\,E_0\,\supset\, E_1\,\supset\, E_2\,\supset\, \cdots\,\supset\, E_j\, \supset\, E_{j+1}\,\supset\, \cdots .$$
In the proof of \cite[Theorem 1.3]{Gi}, it is shown that $E_n$ has an $\cO_Y$--module structure with respect to which $E_n$ is a free $\cO_Y$--module
and $F^*E_{n+1}$ is canonically isomorphic to $E_{n}$.
For an open set $U$ of $Y_T$, we define a similar sequence of subsheaves of the locally free $\scrE_T$:
\[
E_{T,n}(U)\,=\, \{s\,\in\, \scrE_T(U)\,\,\big\vert\,\, D(s)\,=\,0\ \ \forall\ \, D\,\in\, \scrD^{<p^n}_{Y_T/T}\, \text{ with }\, D(1)\,=\,0\},
\]
where $\scrD_{Y_T/T}$ is the sheaf of algebra of relative differential operators for the projection $Y_T\, \longrightarrow\, T$.
Since the differential operators in $\scrD_{Y_T/T}$ are $\cO_T$--linear, the sheaves $E_{T,n}$ have an $\cO_T$--linear structure as well. Moreover, the inverse image 
of $E_{T,n}$ along $Y\,\longrightarrow\, Y_T$ tensored with $k$ over $\cO_T$ is the sheaf $E_n$. Also, the $\cO_Y$--module structure on $E_n$ induces an $\cO_{Y_T}$--module structure on 
$E_{T,n}$. Since $E_n$ is a locally free $\cO_Y$--module, it follows that $E_{T,n}$ is also a locally free $\cO_{Y_T}$--module. Finally, $F^*E_{T,n+1}$ is isomorphic to
$E_{T,n}$ where $F$ is the absolute Frobenius morphism.

Using Theorem \ref{etale-main}, by shrinking $T$ if necessary, we may also assume that for all closed points $t\,\in 
\, T$, the homomorphism $$\pi_1^{et}(f_t)\ :\ \pi_1(Y_t)\ \longrightarrow\ \pi_1(X_t)$$ induced by
$f_t\,:\,Y_t\,\longrightarrow\, X_t$ is an isomorphism.

Now for a given $n$, by \cite[Lemma 7]{shatz} (also see \cite[Theorem 5]{nitsure}) there exists an open dense 
subset $U$ of $T$ such that the Harder--Narasimhan filtration of $f_*E_n$ is compatible with the
Harder--Narasimhan filtration of the fibers 
$(f_t)_*E_{t,n}$ for all $t\,\in\, U$, where $E_{t,n}$ is the restriction of $E_{T,n}$ to $Y_t$.
Consequently, there is a Harder--Narasimhan 
filtration of the family of vector bundles $(f_T)_*E_{U,n}$. Let $F_{U,n}$ be the maximal degree $0$ subsheaf of 
$(f_U)_*E_{U,n}$. Then the restriction $F_{t,n}$ of $F_{U,n}$ is the maximal degree $0$ subsheaf of $(f_t)_*E_{t,n}$.

Now we apply Theorem \ref{main-basecase} to $f_t\,:\,Y_t\,\longrightarrow\, X_t$ and the stratified vector bundle 
$\{E_{t,m}\}$ to conclude that $E_{t,n}\,=\,f_t^*G_{t,n}$ for some $\{G_{t,n}\}$ in $\Vect^{str}(X_t)$, for all 
$t\,\in\, U$.

By the projection formula, we obtain
\begin{equation} \label{eq5}
 (f_t)_*E_{t,n}\,=\,(f_t)_*f_t^*G_{t,n}\,=\,G_{t,n}\otimes_{\cO_{X_t}} (f_t)_*\cO_{Y_t}
\end{equation}

Also, the genuine ramification of $f_t$ implies that the maximal degree zero subsheaf of
$(f_t)_*\cO_{Y_t}$ 
is $\cO_{X_t}$. This and \eqref{eq5} together imply that $G_{t,n}$ is the maximal degree zero subsheaf of
$(f_t)_*E_{t,n}$. Consequently, we obtain that
$$G_{t,n}\ =\ F_{t,n}, \ \ \, E_{t,n}\ =\ f_t^*F_{t,n}$$ for all $t\,\in\, U$, and hence it follows that
$E_{U,n}\,=\, (f_U)^*F_{U,n}$. Restricting to the generic 
point of $U$ yields that $E_n\,=\,f^*F_n$.

So we obtain a stratified vector bundle $\scrF\,=\,\{F_n\}$ on $X$ such that $f^*\scrF\,=\,\scrE$.
Therefore, the functor $$f^*\ :\ \Vect^{str}(X)\ \longrightarrow\ \Vect^{str}(Y)$$ is essentially surjective. That
it is fully faithful is already 
proved in Corollary \ref{basecase} without any assumption on the base field. This completes the proof
of Theorem \ref{main}.

\begin{remark}\label{surjective}
Note that the above argument, together with the proof of Theorem \ref{main-basecase}, also shows that if
$f\,:\,Y\,\longrightarrow\, X$ is a genuinely ramified map,
then the induced homomorphism $\pi_1^{str}\,:\,\pi_1^{str}(Y)\,\longrightarrow\, \pi_1^{str}(X)$ is surjective
(see \cite[p. 139, Proposition 2.21]{DM}).
\end{remark}

\section{Proof of Proposition \ref{fullyfaithful}}\label{ffp}

Let $f\,:\,Y\,\longrightarrow\, X$ be a finite generically smooth morphism between two irreducible 
projective varieties of the same dimension. Denote by ${\rm Aut}(Y/X)$ the group of automorphisms of
$Y$ over the identity map of $X$. The morphism $f$ will be called \textit{Galois} if there is a
reduced finite subgroup $\Gamma\, \subset\, {\rm Aut}(Y/X)$ such that $X\,=\, Y/\Gamma$.

Note that we have ${\mathcal O}_X \,\subset\, f_*{\mathcal O}_Y$, because $f^*{\mathcal O}_X\,=\, {\mathcal O}_Y$
(use adjunction).

\begin{lemma}\label{ffl1}
Let $f\,:\,Y\,\longrightarrow\, X$ be a generically smooth morphism between irreducible smooth projective varieties
of the same dimension. Assume that $f$ is Galois of degree $d$. Then
$$
f^*((f_*{\mathcal O}_Y)/{\mathcal O}_X) \ \, \subseteq\ \, {\mathcal O}^{\oplus (d-1)}_X
$$
\end{lemma}

\begin{proof}
The proof of the lemma is exactly identical to the proof of \cite[p.~12831, Proposition 3.3]{BP}. 
In \cite[Proposition 3.3]{BP} it is assumed that $\dim X\,=\, 1\,=\, \dim Y$, because \cite{BP}
is entirely dedicated to curves but this assumption that $\dim X\,=\, 1\,=\, \dim Y$ is not used in
the proof of \cite[Proposition 3.3]{BP}.
\end{proof}

Let $f\,:\,Y\,\longrightarrow\, X$ be a genuinely ramified map between irreducible smooth projective varieties.
As in Lemma \ref{ffl1}, assume that $f$ is a Galois map of degree $d$. Let
\begin{equation}\label{ega}
\Gamma\ :=\ \text{Gal}(f)
\end{equation}
be the Galois group of $f$, so $X\,=\, Y/\Gamma$. For any element $\sigma\, \in\, \Gamma$, let
\begin{equation}\label{ecs}
Y_\sigma \ :=\ \{(y, \, \sigma(y))\, \in\, Y\times_X Y\,\, \big\vert\,\, y\, \in\, Y\}\ \subset\ Y\times_X Y
\end{equation}
be the irreducible component of $Y\times_X Y$.

\begin{lemma}\label{ffl2}
There is an ordering of the elements of the group $\Gamma$ in \eqref{ega}
$$
\Gamma\ =\ \{\gamma_1,\, \cdots,\, \gamma_d\},
$$
and a self--map $\eta\ :\ \{1,\, \cdots,\, d\}\ \longrightarrow\ \{1,\, \cdots,\, d\}$, such that
the following four statements hold:
\begin{enumerate}
\item $\gamma_1\,=\, e$ (the identity element of the group $\Gamma$),

\item $\eta(1)\, =\, 1$,

\item $\eta(j) \, <\, j$ for all $j\, \in\, \{2,\, \cdots ,\, d\}$, and

\item $Y_{\gamma_j}\bigcap Y_{\gamma_{\eta(j)}}\, \not=\, \emptyset$ (see \eqref{ecs} for notation).
\end{enumerate}
\end{lemma}

\begin{proof}
The proof of the lemma is exactly identical to the proof of \cite[p.~12835, Lemma 3.4]{BP}.
The proof of \cite[Lemma 3.4]{BP} is combinatorial and only the connectedness of $Y\times_X Y$ is used. Note that
this is true even when $\dim X \,>\,1$ (see \cite[p.~6, Theorem 2.4(3)]{BDP}). The assumption that
$\dim X\,=\, 1\,=\, \dim Y$ is not used in the rest of the proof of \cite[Lemma 3.4]{BP}.
\end{proof}

\begin{remark}\label{remi}
The intersection $Y_{\gamma_j}\bigcap Y_{\gamma_{\eta(j)}}$ in the fourth statement of Lemma \ref{ffl2}
coincides with the fixed--point locus for the element $(\gamma_j)^{-1}\gamma_{\eta(j)}\, \in\, \Gamma\,=\,
\text{Gal}(f)$. Since $X$ and $Y$ are both smooth, from purity of branch locus we know that
$Y_{\gamma_j}\bigcap Y_{\gamma_{\eta(j)}}$ is a divisor on $Y$.
\end{remark}

The following lemma constitutes a key input in the proof of Proposition \ref{fullyfaithful}.

\begin{lemma}\label{ffl3}
Let $f\,:\,Y\,\longrightarrow\, X$ be a genuinely ramified map between irreducible smooth projective
varieties. Assume that $f$ is Galois of degree $d$. Then there are line bundles
$$
{\mathcal L}_j\ \subsetneq \ {\mathcal O}_Y, \ \ \, 1\, \leq\, j\, \leq\, d-1,
$$
such that
$$
f^*((f_*{\mathcal O}_Y)/{\mathcal O}_X) \ \, \subseteq\ \, \bigoplus_{j=1}^{d-1} {\mathcal L}_j.
$$
\end{lemma}

\begin{proof}
The proof is rather identical to the proof of \cite[p.~12837, Lemma 3.5]{BP}. The details are given for the
benefit of the reader.
As in \eqref{ega}, the Galois group $\text{Gal}(f)$ is denoted by $\Gamma$.
Let $\nu\, :\, \widetilde{Y\times_X Y}\, \longrightarrow\, Y\times_X Y$ be the normalization of
$Y\times_X Y$. For $i\,=\, 1,\, 2$, let
\begin{equation}\label{epi}
\widetilde{\pi}_i \ :=\ \pi_i \circ \nu
\ :\ \widetilde{Y\times_X Y}\ \longrightarrow\ Y
\end{equation}
be the composition of maps, where $\pi_i\, :\, Y\times_X Y\, \longrightarrow\, Y$ is the projection
to the $i$--th factor. We have an isomorphism
\begin{equation}\label{epi-1}
Y\times\Gamma\ \longrightarrow\ \widetilde{Y\times_X Y}, \ \ \, (y,\, \gamma)\ \longmapsto\ (y,\, \gamma(y)).
\end{equation}
So, for the map $\widetilde{\pi}_1$ in \eqref{epi},
\begin{equation}\label{ep2}
(\widetilde{\pi}_1)_* {\mathcal O}_{\widetilde{Y\times_X Y}}\ =\ {\mathcal O}_Y\otimes_k k[\Gamma].
\end{equation}

Using adjunction, the identity map $\pi^*_1 {\mathcal O}_Y\,=\, {\mathcal O}_{Y\times_X Y}
\, \xrightarrow{\,\,\, {\rm Id}_{{\mathcal O}_{Y\times_X Y}}\,\,\,\,} \, {\mathcal O}_{Y\times_X Y}$
produces a homomorphism
$$
\zeta\ :\ {\mathcal O}_Y \ \longrightarrow\ (\pi_1)_* {\mathcal O}_{Y\times_X Y}
$$
On the other hand, since $\pi_1\circ \nu\,=\, \widetilde{\pi}_1$ (see \eqref{epi}), and $\nu$
is surjective, we have an
injective homomorphism
\begin{equation}\label{vp}
\varphi\ :\ (\pi_1)_* {\mathcal O}_{Y\times_X Y}\ \longrightarrow\ (\widetilde{\pi}_1)_*
{\mathcal O}_{\widetilde{Y\times_X Y}}.
\end{equation}
Let
\begin{equation}\label{exi}
\xi\ :=\ \varphi \circ\zeta\ :\ {\mathcal O}_Y \ \longrightarrow\ (\widetilde{\pi}_1)_*
{\mathcal O}_{\widetilde{Y\times_X Y}} \ =\ {\mathcal O}_Y\otimes_k k[\Gamma]
\end{equation}
be the composition of homomorphisms (see \eqref{ep2}).

The ordering, in Lemma \ref{ffl2}, of the elements of $\Gamma$ produces an isomorphism of $k[\Gamma]$ with $k^{\oplus d}$.
Consequently, from \eqref{ep2} we have
\begin{equation}\label{z1}
(\widetilde{\pi}_{1})_* {\mathcal O}_{\widetilde{Y\times_X Y}}\ =\
{\mathcal O}_Y\otimes_k k[\Gamma]\,=\, {\mathcal O}^{\oplus d}_Y\, .
\end{equation}

Let
$$
\Phi\, :\, (\widetilde{\pi}_{1})_* {\mathcal O}_{\widetilde{Y\times_X Y}}\,=\, {\mathcal O}^{\oplus d}_Y
\, \longrightarrow\, {\mathcal O}^{\oplus d}_Y \,=\,
(\widetilde{\pi}_{1})_* {\mathcal O}_{\widetilde{Y\times_X Y}}
$$
be the homomorphism defined by
\begin{equation}\label{re4}
(f_1,\, f_2,\, \cdots, \, f_d)\, \longmapsto\,
(f_1-f_{\eta(1)},\, f_2-f_{\eta(2)},\, \cdots, \, f_d-f_{\eta(d)})\, ,
\end{equation}
where $\eta$ is the map in Lemma \ref{ffl2}; in other words, the $j$-th
component of $\Phi(f_1,\, f_2,\, \cdots, \, f_d)$ is $f_j-f_{\eta(j)}$. The image
$$
{\mathcal F}\, :=\, \Phi({\mathcal O}^{\oplus d}_Y)\, \subset\,
{\mathcal O}^{\oplus d}_Y \,=\,
(\widetilde{\pi}_{1})_* {\mathcal O}_{\widetilde{Y\times_X Y}}
$$
is a trivial subbundle of rank $d-1$; the first component of $\Phi(f_1,\, f_2,\, \cdots, \, f_d)$
vanishes identically, because $\eta(1)\,=\,1$. More precisely,
\begin{equation}\label{zf2}
{\mathcal F}\,=\, {\mathcal O}^{\oplus (d-1)}_Y\, \subset\, {\mathcal O}^{\oplus d}_Y
\,=\, (\widetilde{\pi}_{1})_* {\mathcal O}_{\widetilde{Y\times_X Y}},
\end{equation}
where ${\mathcal O}^{\oplus (d-1)}_Y$ is the subbundle of ${\mathcal O}^{\oplus d}_Y$
generated by all $(f_1,\, f_2,\, \cdots, \, f_d)$ such that $f_1\,=\,0$.

{}From \eqref{zf2} it follows immediately that
\begin{equation}\label{cf2}
(\widetilde{\pi}_{1})_* {\mathcal O}_{\widetilde{Y\times_X Y}}\,=\,
{\mathcal O}^{\oplus d}_Y\,=\, {\mathcal F}\oplus \xi({\mathcal O}_Y)\, ,
\end{equation}
where $\xi({\mathcal O}_C)$ is the homomorphism in \eqref{exi}.

We have the commutative diagram
\begin{equation}\label{d1}
\xymatrix{
\widetilde{Y\times_X Y} \ar@/_/[ddr]_-{\widetilde{\pi}_1} \ar[dr]^-\nu \ar@/^/[drrr]^-{\widetilde{\pi}_2} & & & \\
& Y\times_X Y \ar[rr]^-{\pi_2} \ar[d]^-{\pi_1} && Y \ar[d]^-f \\
& Y \ar[rr]^-f && X
}
\end{equation}
By flat base change \cite[p.~255, Proposition 9.3]{Ha},
\begin{equation}\label{f1}
f^* (f_* {\mathcal O}_Y)\ \cong\ ({\pi_1})_* (\pi^*_2 {\mathcal O}_Y) \ =\ ({\pi_1})_* {\mathcal O}_{Y\times_X Y}.
\end{equation}

{}From \eqref{cf2} and \eqref{f1} we get an injective homomorphism
of coherent sheaves
\begin{equation}\label{cf3}
\Psi\ :\ f^* ((f_* {\mathcal O}_Y) / {\mathcal O}_X)\ \longrightarrow \ {\mathcal F}\ =\
{\mathcal O}^{\oplus (d-1)}_Y.
\end{equation}
Note that $\Psi$ is an isomorphism over the open subset
of $C$ where the map $f$ is a submersion.

Consider the map $\eta$ in Lemma \ref{ffl2}. For every $1\,\leq\, i\, \leq\, d-1$, define
\begin{equation}\label{zi-1}
D_i\ := \ Y_{\gamma_{i+1}}\bigcap Y_{\gamma_{\eta(i+1)}}
\end{equation}
(see \eqref{ecs}); from the fourth property in Lemma \ref{ffl2} and Remark \ref{remi} it follows
that $D_i$ is a nonzero effective divisor on $Y$. So
\begin{equation}\label{zi}
D^0_i \ :=\ \{y\, \in\, Y\, \,\big\vert\,\, (y,\, \gamma_{i+1}(y))\, \in\, D_i\} \ \subset\ Y
\end{equation}
is a nonzero effective divisor on $Y$. Let
$$
{\mathcal L}_i\ :=\ {\mathcal O}_Y(-D^0_i) \ \subset\ {\mathcal O}_Y 
$$
be the lie bundle on $Y$ given by the divisor $-D^0_i$.

For every $1\,\leq\, i\, \leq\, d-1$, let
\begin{equation}\label{pj}
P_i\, :\, {\mathcal O}^{\oplus (d-1)}_Y\, \longrightarrow\, {\mathcal O}_Y
\end{equation}
be the natural projection to the $i$-th factor.
Consider the composition of homomorphisms $P_i\circ\Psi$, where $P_i$ and
$\Psi$ are constructed in \eqref{pj} and \eqref{cf3} respectively. We will show that
$P_i\circ\Psi$ vanishes when restricted to
$D^0_i$ in \eqref{zi}. To see this, for any $1\,\leq\, j\, \leq\, d$, let
$$
\widehat{P}_j\, :\, {\mathcal O}^{\oplus d}_Y\, \longrightarrow\, {\mathcal O}_Y
$$
be the natural projection to the $j$-th factor. Recall the homomorphism $\Phi$ constructed
in \eqref{re4}. If $(f_1,\, f_2,\, \cdots, \, f_d)$ in \eqref{re4} actually lies in the image
of $({\pi_1})_* {\mathcal O}_{Y\times_X Y}$ by the inclusion map $\varphi$ in \eqref{vp}, then from
\eqref{zi-1} we have
\begin{equation}\label{re6}
(\widehat{P}_{i+1}\circ \Phi)(f_1,\, f_2,\, \cdots, \, f_d)(y,\, \gamma_{i+1}) \ =\
f_{i+1}(y,\, \gamma_{i+1})-f_{\eta(i+1)}(y,\, \gamma_{\eta(i+1)})\ =\ 0,
\end{equation}
where $y\, \in\, D^0_i$ (see \eqref{zi}), and also
$$
(\widehat{P}_{i+1}\circ \Phi)(f_1,\, f_2,\, \cdots, \, f_d)(y,\, \gamma_{\eta(i+1)}) \ =\
f_{i+1}(y,\, \gamma_{i+1})-f_{\eta(i+1)}(y,\, \gamma_{\eta\circ\eta(i+1)})\ =\ 0
$$
for $y\, \in\, D^0_i$. From \eqref{re6} it follows that
$P_i\circ\Psi$ vanishes when restricted to $D^0_i$, where $\Psi$ and $P_i$ are constructed in
\eqref{cf3} and \eqref{pj} respectively.

Since $P_i\circ\Psi$ vanishes when restricted to the divisor $D^0_i$, we have
\begin{equation}\label{cf4}
P_i\circ\Psi(f^* ((f_* {\mathcal O}_Y) / {\mathcal O}_X))\ \subset\ {\mathcal L}_i\,=\, {\mathcal O}_Y(-D^0_i)
\ \subset\ {\mathcal O}_Y .
\end{equation}
{}From \eqref{cf3} and \eqref{cf4} it follows immediately that
$$
f^* ((f_* {\mathcal O}_Y)/{\mathcal O}_X)\ \hookrightarrow\ \bigoplus_{i=1}^{d-1} {\mathcal L}_i .
$$
This completes the proof of the proposition.
\end{proof}

\begin{lemma}\label{ffl4}
Let $f\,:\,Y\,\longrightarrow\, X$ be a genuinely ramified map between irreducible smooth projective
varieties. For any semistable vector bundle $V$ on $X$,
$$
\mu_{\rm max}(V\otimes ((f_* {\mathcal O}_Y)/{\mathcal O}_X))\ <\ \mu(V).
$$
\end{lemma}

\begin{proof}
In view of Lemma \ref{ffl3}, the proof is exactly identical to the proof of \cite[p.~12840, Lemma 4.1]{BP}.
Note that in Lemma \ref{ffl3} we have
$$
\text{degree}({\mathcal L}_j)\ <\ 0
$$
for all $1\, \leq\, j\, \leq\, d-1$, because ${\mathcal L}_j\ \subsetneq \ {\mathcal O}_Y$ and
${\mathcal L}_j$ is locally free. More, precisely, $-\text{degree}({\mathcal L}_j)$ coincides with the
degree of the divisor whose ideal sheaf is ${\mathcal L}_j$.
\end{proof}
 
\begin{proof}[{Proof of Proposition \ref{fullyfaithful}}]
The proof is very similar to the proof of \cite[p.~12844, Lemma 4.3]{BP}.
The details are given for the benefit of the reader.

Using the projection formula, and the fact that $f$ is a finite map, we have
$$
H^0(Y,\, {\rm Hom}(f^*V,\, f^*W)) \ \cong\ H^0(X,\, f_*{\rm Hom}(f^*V,\, f^*W))
\ \cong\ H^0(X,\, f_*f^*{\rm Hom}(V,\, W))
$$
\begin{equation}\label{j1}
\cong\ H^0(X,\, {\rm Hom}(V,\, W)\otimes f_*{\mathcal O}_Y)
\,\cong\,H^0(X,\, {\rm Hom}(V,\, W\otimes f_*{\mathcal O}_Y)).
\end{equation}
Let
$$
0\,=\, B_0\, \subset\, B_1\, \subset\, \cdots \, \subset\, B_{m-1}\, \subset\, B_m\,=\,
W\otimes ((f_*{\mathcal O}_Y)/{\mathcal O}_X)
$$
be the Harder--Narasimhan filtration of $W\otimes ((f_*{\mathcal O}_Y)/{\mathcal O}_X)$.
Since $W$ is semistable, and $f$ is genuinely ramified, from
Lemma \ref{ffl4} we know that 
$$
\mu(B_i/B_{i-1})\ \leq\ \mu(B_1)\ =\ \mu_{\rm max}(W\otimes ((f_*{\mathcal O}_Y)/{\mathcal O}_X))
\ <\ \mu(W)\ =\ \mu(V)
$$
for all $1\, \leq\, i\, \leq\, m$. Since both $V$ and $B_i/B_{i-1}$ are semistable, and
$\mu(B_i/B_{i-1})\ \leq\ \mu(V)$, we have
$$
H^0(X,\, {\rm Hom}(V,\, B_i/B_{i-1}))\,=\, 0
$$
for all $1\, \leq\, i\, \leq\, m$. This implies that
\begin{equation}\label{h1}
H^0(X,\, {\rm Hom}(V,\, W\otimes ((f_*{\mathcal O}_Y)/{\mathcal O}_X)))\,=\ 0.
\end{equation}

Now consider the short exact sequence of sheaves
$$
0 \,\longrightarrow\,{\rm Hom}(V,\, W)\,\longrightarrow\, {\rm Hom}(V,\, W\otimes f_*{\mathcal O}_Y)
\,\longrightarrow\, {\rm Hom}(V,\, W\otimes ((f_*{\mathcal O}_Y)/{\mathcal O}_X))\,\longrightarrow\, 0,
$$
and the corresponding exact sequence of cohomologies
\begin{equation}\label{h2}
0\ \longrightarrow\ H^0(X,\, {\rm Hom}(V,\, W))\ \longrightarrow\ H^0(X,\, {\rm Hom}(V,\, W\otimes f_*{\mathcal
O}_Y))
\end{equation}
$$
\ \longrightarrow\ H^0({\rm Hom}(V,\, W\otimes ((f_*{\mathcal O}_Y)/{\mathcal O}_X))).
$$
Combining \eqref{h1} and \eqref{h2} it follows that
$$
H^0(X,\, {\rm Hom}(V,\, W))\ = \ H^0(X,\, {\rm Hom}(V,\, W\otimes f_*{\mathcal
O}_Y)).
$$
{}From this and \eqref{j1} it follows that
$$
H^0(X,\, {\rm Hom}(V,\,W))\ =\ H^0(Y,\, {\rm Hom}(f^*V,\, f^*W)).$$
This completes the proof.

\end{proof}

\section*{Acknowledgements}

We thank the referee for helpful comments. The first author is partially supported by a J. C.
Bose Fellowship (JBR/2023/000003). The second author is partially supported by a SERB Core grant 
(CRG/2023/006248).


\begin{thebibliography}{ZZZZ}

\bibitem[BO]{BO} P. Berthelot and A. Ogus, {\it Notes on crystalline cohomology},
Princeton University Press, Princeton NJ, University of Tokyo Press, Tokyo, 1978.

\bibitem[BP]{BP} I. Biswas and A. J. Parameswaran, Ramified covering maps and stability of
pulled back bundles, {\it Int. Math. Res. Not.}, (2022), no. 17, 12821--12851.

\bibitem[BDP]{BDP} I. Biswas, S. Das, and A. J. Parameswaran, Genuinely ramified maps and stable vector
bundles, {\it Internat. J. Math.} {\bf 33} (2022), Paper No. 2250039, 14.

\bibitem[DM]{DM} P. Deligne and J. S. Milne, Tannakian Categories, {\it Hodge cycles, motives, and Shimura
varieties}, by P. Deligne, J. S. Milne, A. Ogus and K.-Y. Shih, pp. 101--228, Lecture Notes in Mathematics, 900,
Springer-Verlag, Berlin-Heidelberg-New York, 1982.

\bibitem[EM]{EM} H. Esnault and V. Mehta, Simply connected projective manifolds in characteristic $p>0$
have no nontrivial stratified bundles, {\it Invent. Math.} {\bf 181} (2010), 449--465.

\bibitem[ES]{ES} H. Esnault and V. Srinivas, A relative version of Gieseker's problem on stratifications
in characteristic $p>0$, {\it Inter. Math. Res. Not.} (2019), 5635--5648.

\bibitem[Gi]{Gi} D. Gieseker, Flat vector bundles and the fundamental group in non-zero characteristics,
{\it Ann. Scuola Norm. Sup. Pisa Cl. Sci.} {\bf 2} (1975), 1--31.

\bibitem[Gr]{Gr} A. Grothendieck, \'El\'ements de g\'eom\'etrie alg\'ebrique. IV. \'Etude locale des
sch\'emas et des morphismes de sch\'emas IV, {\it Inst. Hautes \'Etudes Sci. Publ. Math.} {\bf 32} (1967),
5--361.

\bibitem[Ha]{Ha} R. Hartshorne, {\it Algebraic geometry}, Graduate Texts in Mathematics, No. 52. Springer-Verlag,
New York-Heidelberg, 1977.

\bibitem[Hr]{Hru} E. Hrushovski, The elementary theory of the Frobenius automorphism, preprint, available at 
http://www.ma.huji.ac.il/$\sim$ehud/FROB.pdf.

\bibitem[LS]{LS} H. Lange and U. Stuhler, Vektorb\"undel auf Kurven und Darstellungen der algebraischen
Fundamentalgruppe, {\it Math. Zeit.} {\bf 156} (1977), 73--83.

\bibitem[Mur]{Mur} J. P. Murre, {\it Lectures on an introduction to {G}rothendieck's theory of the
 fundamental group,} Notes by S. Anantharaman, Tata Institute of Fundamental Research Lectures on Mathmatics,
 No 40, 1967.

\bibitem[Ni]{nitsure} N. Nitsure, Schematic Harder-Narasimhan stratification, {\it
Internat. J. Math.} {\bf 22} (2011), 1365--1373.

\bibitem[No1]{No1}M. V. Nori, On the representations of the fundamental group,
{\it Compositio Math.} {\bf 33} (1976), 29--41.

\bibitem[No2]{No2} M. V. Nori, The fundamental group scheme, {\it Proc. Indian Acad. Sci. Math. Sci.} {\bf 91}
(1982), 73--122.

\bibitem[SGA1]{SGA} A. Grothendieck, {\it Rev\^etements \'etales et groupe fondamental}, {\it SGA 1}. 

\bibitem[Sa]{Sa} J. P. P. dos Santos, Fundamental group schemes for stratified sheaves,
{\it J. Algebra} {\bf 317} (2007), 691--713.

\bibitem[SR]{SR} N. Saavedra Rivano, {\it Cat\'egories Tannakiennes}, Lecture Notes in Mathematics, Vol. 265,
Springer-Verlag, Berlin, Heidelberg, 1972.

\bibitem[Sh]{shatz} S. Shatz, The decomposition and specialization of algebraic families of vector bundles,
{\it Compositio Math.} {\bf 35} (1977), 163--187.

\bibitem[Sim]{Simpson} C.T. Simpson, Moduli of representations of the fundamental groups of a smooth projective 
variety I, {\it Inst. Hautes \'Etudes Sci. Publ. Math.} {\bf 79} (1994), 47--129.

\bibitem[SP]{SP} {\it Stacks Project}, http://stacks.math.columbia.edu/.

\bibitem[Sun]{Sun} X. Sun, {Stratified bundles and representation spaces}, {\it Adv. Math.} {\bf 345} (2019), 767--783.

\bibitem[Var]{Var} Y. Varshavsky, Intersection of a correspondence with a graph of Frobenius,
{\it J. Algebraic Geom.} {\bf 27} (2018), 1--20.
\end{thebibliography}
\end{document}